\documentclass[11pt,letterpaper]{article}
\usepackage{amsthm}
\usepackage{amsmath}
\usepackage{amsfonts}
\usepackage{latexsym}
\usepackage{geometry}
\usepackage{enumerate}
\usepackage{csquotes}

\emergencystretch=\maxdimen
\newtheorem{Theorem}{Theorem}[section]
\newtheorem{Proposition}[Theorem]{Proposition}
\newtheorem{Corollary}[Theorem]{Corollary}
\newtheorem{Lemma}[Theorem]{Lemma}
\theoremstyle{definition}
\newtheorem{Definition}[Theorem]{Definition}

\title{On the noncollapsedness of positively curved Type I ancient Ricci flows}
\author{Liang Cheng and Yongjia Zhang}
\numberwithin{equation}{section}
\begin{document}
\maketitle

\begin{abstract}
In this article, we study complete Type I ancient Ricci flows with positive sectional curvature. Our main results are as follows: in the complete and noncompact case, all such ancient solutions must be noncollapsed on all scales; in the closed case, if the dimension is even, then all such ancient solutions must be noncollapsed on all scales. This furthermore gives a complete classification for three-dimensional noncompact Type I ancient solutions without assuming the noncollapsing condition.
\end{abstract}

\section{Introduction to the Main Results}

The study of ancient solutions to the Ricci flow, ever since Hamilton had published his program \cite{Ha1}, has been an important topic in the field of Ricci flow. Ancient solutions are Ricci flows whose existing intervals extend to negative infinity. They are of great importance because they usually arise as blow-up limits at finite-time singularities of the Ricci flow, and to this kind of blow-up limits, a term not improper, ``singularity models'', is assigned.

Perelman \cite{P} proved that a Ricci flow on a closed manifold cannot become locally collapsed within finite time. Subsequently one may conclude that every singularity model must be $\kappa$-noncollapsed on all scales. This precisely means the following.

\begin{Definition}[$\kappa$-noncollapsing]
A Ricci flow $(M^n,g(t))$ is called $\kappa$-noncollapsed on all scales, where $\kappa$ is a positive number, if for any point $(x,t)$ in space-time and any positive scale $r$, it holds that $\displaystyle\operatorname{Vol}_{g(t)}\big(B_{g(t)}(x,r)\big)\geq\kappa r^n$ whenever $R\leq r^{-2}$ on $B_{g(t)}(x,r)$. Here $R$ stands for the scalar curvature
\end{Definition}

The noncollapsing notion defined above is sometimes called the \emph{strong} noncollapsing in the literature. The \emph{weak} noncollapsing notion is defined similarly with only the ``whenever $R\leq r^{-2}$ on $B_{g(t)}(x,r)$'' statement replaced by ``whenever $|Rm|\leq r^{-2}$ on $B_{g(t)}(x,r)\times [t-r^2,t]$''. It is known that for an ancient solution with bounded and nonnegative curvature operator, the weak noncollapsing is equivalent to the strong noncollapsing condition, with possibly a different $\kappa$. The noncollapsing condition which we use throughout this paper is the strong one. We remark that these two notions are not equivalent in general. For instance, a closed nonflat and Ricci-flat (static) Ricci flow is weakly noncollapsed but not strongly noncollapsed.

Since, according to Hamilton \cite{Ha1}, ancient solutions are critical to the understanding of the singularity formation in the Ricci flow (see, for instance, Perelman's proof of the canonical neighborhood theorem \cite{P}), it makes sense to assume the noncollapsing condition when studying ancient solutions. With this assumption, many groundbreaking works are done, and the most outstanding one is of Perelman \cite{P}. See also \cite{B1}, \cite{B2}, \cite{B}, and \cite{L}, etc., to list but a few.

It is well-known that not all ancient solutions are noncollpased. But what if some further conditions are added? Concerning this a question is proposed in \cite{CLN}:
\begin{quotation}
\noindent\textbf{Problem 9.41.} \emph{Are nonflat Type I ancient solutions with nonnegative curvature operator $\kappa$-solutions?} 
\end{quotation}
Recall that an ancient solution $(M,g(t))_{t\in(-\infty,w)}$ is called \emph{Type I} if
\begin{eqnarray*}
\limsup_{t\rightarrow-\infty}|t|\big|Rm_{g(t)}\big|<\infty.
\end{eqnarray*}
Without any further qualification, the answer to the above question is obviously ``no'', either in the closed case or in the complete and noncompact case. One may immediately think of $\mathbb{S}^{n-1}\times\mathbb{S}^1$ or $\mathbb{S}^{n-2}\times\mathbb{R}\times\mathbb{S}^1$ as counterexamples. There are also much more sophisticated counterexamples constructed. For instance, Fateev \cite{F} constructed an ancient solution on the Hopf fiber bundle, and Bakas-Kong-Ni \cite{BKN} generalized this construction to all odd-dimensional spheres; all these ancient solutions are Type I, collapsed, and with positive curvature operator.

In this article, we give a relatively satisfactory answer to Problem 9.41 in \cite{CLN} as quoted above. First of all, to rule out the possibility of a compact flat factor, with which the ancient solution is always collapsed, we would like to assume that the sectional curvature is strictly positive. This condition also largely simplifies the geometry in the complete and noncompact case, since the underlying manifold must be diffeomorphic to the Euclidean space by the Gromoll-Meyer theorem.

\begin{Theorem}  \label{MainTheorem1}
Let $(M^n,g(t))_{t\in(-\infty,w)}$, where $0<w\leq\infty$, be a complete and noncompact Type I ancient solution with positive sectional curvature. Then $(M^n,g(t))_{t\in(-\infty,0]}$ is $\kappa$-noncollapsed on all scales for some $\kappa>0$.
\end{Theorem}

The Type I, closed, and collapsed examples constructed in \cite{BKN} are only in odd dimensions, while the cases of even dimensions are yet open. The following theorem shows that there are no such collapsed examples in even dimensions.

\begin{Theorem}\label{MainTheorem2}
Let $(M^n,g(t))_{t\in(-\infty,w)}$, where $n$ is an even number and $0<w\leq\infty$, be a closed Type I ancient solution with positive sectional curvature. Then $(M^n,g(t))_{t\in(-\infty,0]}$ is $\kappa$-noncollapsed on all scales for some $\kappa>0$.
\end{Theorem}

The critical observations applied to the proofs of Theorem \ref{MainTheorem1} and Theorem \ref{MainTheorem2} are some injectivity radius estimates resulted from the Gromoll-Meyer theorem and the Klingenberg theorem. These classical theorems imply that an ancient solution as described in Theorem \ref{MainTheorem1} or Theorem \ref{MainTheorem2} has a Type I injectivity radius lower bound. This is sufficient to conclude the existence of an asymptotic shrinker, which in turn implies noncollapsedness; the authors used a similar idea in \cite{ChZ} to prove the noncollapsedness of a more general type of ancient Ricci flows---the locally uniformly Type I ancient solutions. Here we emphasize that the Type I injectivity radius lower bound itself does not directly imply the noncollapsedness; see section 2 for more details concerning this point.

An immediate application of Theorem \ref{MainTheorem1} is the following classification of three-dimensional Type I ancient solutions without assuming the noncollapsing condition. This is a generalization of \cite{Z1} (or \cite{H}). In consequence, all noncompact three-dimensional Type I ancient solutions must be noncollapsed.

\begin{Corollary} \label{3d}
A three-dimensional noncompact Type I ancient solution must be the standard cylinder $\mathbb{S}^2\times\mathbb{R}$, $(\mathbb{S}^2\times\mathbb{R})/\mathbb{Z}_2$, or $\operatorname{\mathbb{R}P}^2\times\mathbb{R}$. Hence it must also be $\kappa$-noncollapsed on all scales for some $\kappa>0$.
\end{Corollary}

In \cite{CLN}, it is asked in Problem 9.40 whether a Type I ancient solution with positive curvature operator is closed. While we are not yet able to give an answer to this question, the following Corollary rules out one possibility of its asymptotic shrinker---the standard cylinder. This also means that if such an ancient solution did exist, then its geometry could not be very simple (though topologically it is the Euclidean space).

\begin{Corollary} \label{noncompact}
A Type I complete and noncompact ancient Ricci flow with nonnegative curvature operator and positive sectional curvature cannot have the standard cylinder $\mathbb{S}^{n-1}\times\mathbb{R}$ as its asymptotic shrinker.
\end{Corollary}

Ni \cite{Ni2} classified all closed Type I ancient and noncollapsed Ricci flows with nonnegative curvature operator. It is noted in \cite{BKN} that, because of the examples therein, the noncollapsing condition in the classification of \cite{Ni2} cannot be dropped. Nonetheless, because the examples in \cite{BKN} are only in odd dimensions, this conclusion is not decisive in even dimensions. Indeed, the noncollapsing condition can be dropped in even dimension, at least for the strictly positive curvature operator case.

\begin{Corollary} \label{compact}
An even-dimensional closed Type I ancient Ricci flow satisfying the strict $\operatorname{PIC}-2$ curvature condition (and in particular, with positive curvature operator) must be a standard shrinking round space form.
\end{Corollary}

Here we remark that, unlike \cite{Ni2}, we are not able to deal with the case when the curvature operator admits a zero eigenvalue, since in this case we can no longer prove the noncollapsedness. For instance, one may think of $\mathbb{S}^1\times\mathbb{S}^{2m+1}$, where the $\mathbb{S}^1$ factor is static and the $\mathbb{S}^{2m+1}$ factor is the standard shrinking sphere.
\\

This paper is organized as follows. In section 2 we use the Gromoll-Meyer theorem and the Klingenberg theorem to derive the Type I injectivity radius lower bound for the ancient solutions in question. In section 3 we review the fact that Perelman's entropy and the Nash entropy converge to the entropy of the asymptotic shrinker. In section 4 we show that the existence of asymptotic shrinker implies noncollapsedness. In section 5 we prove all the corollaries.

\emph{Acknowledgment.} The second author would like to thank Professor Jiaping Wang, Professor Lei Ni, and Professor Bennett Chow for many helpful discussions.

\section{The injectivity radius}

The injectivity radius estimates are provided by the following classical theorems of Gromoll-Meyer and Klingenberg. 
\begin{Proposition}[Gromoll-Meyer; c.f. Theorem 1.168 in \cite{CLN}]\label{Gromoll-Meyer}
Let $(M^n,g)$ be a complete and noncompact Riemannian manifold satisfying $0<\operatorname{sec}\leq K$, where $K$ is a positive number. Then the injectivity radius of $(M^n,g)$ satisfies
\begin{eqnarray*}
\operatorname{inj}(g)\geq\frac{\pi}{\sqrt{K}}.
\end{eqnarray*}
Moreover, $(M^n,g)$ is diffeomorphic to the standard Euclidean space.
\end{Proposition}

\begin{Proposition}[Klingenberg; c.f.  Theorem 1.115 in \cite{CLN}]\label{Klingenberg}
Let $(M^n,g)$ be an even-dimensional closed orientable manifold satisfying $0<\operatorname{sec}\leq K$, where $K$ is a positive number. Then the injectivity radius of $(M^n,g)$ satisfies
\begin{eqnarray*}
\operatorname{inj}(g)\geq\frac{\pi}{\sqrt{K}}.
\end{eqnarray*}
\end{Proposition}

The results above imply that the ancient solutions in question have Type I injectivity radii lower bounds.

\begin{Lemma}\label{inj-radius}
Let $(M^n,g(t))_{t\in(-\infty,w)}$ be an ancient solution as described in either Theorem \ref{MainTheorem1} or Theorem \ref{MainTheorem2}. Then there exists a constant $c>0$, such that the injectivity radius of $g(t)$ satisfies
\begin{eqnarray*}
\operatorname{inj}(g(t))\geq c\sqrt{|t|},
\end{eqnarray*}
for all $t\in(-\infty,0)$.
\end{Lemma}
\begin{proof}
By the Type I and the positive sectional curvature conditions, we have that there exists $C>0$, such that
\begin{eqnarray*}
0<\operatorname{sec}(g(t))\leq\frac{C}{|t|},
\end{eqnarray*}
for all $t\in(-\infty,0)$. The lemma then follows from Proposition \ref{Gromoll-Meyer} and Proposition \ref{Klingenberg}. Note that though Proposition \ref{Klingenberg} requires orientability, which is not assumed in the statement of Theorem \ref{MainTheorem2}, yet one may always consider the double cover if necessary.
\end{proof}

Here we remark again that this Type I injectivity radius lower bound does not imply $\kappa$-noncollapsedness on all scales, since at some point on the manifold, the curvature could decay faster than Type I. Nevertheless, this estimate is sufficient for the existence of an asymptotic shrinker, which in turn implies $\kappa$-noncollapsedness on all scales.
\\

\section{The Asmptotic Shrinker}

Perelman \cite{P} and Naber \cite{Na} proved the existence of the asymptotic shrinker for ancient solutions under the assumptions of nonnegative curvature operator and of Type I curvature bound, respectively. They both also assumed the noncollapsing condition. However, it turns out that the only place where they applied this condition was to obtain an injectivity radius lower bound for a blow-down sequence, and they need this injectivity radius lower bound only at base points to conclude the convergence. This, of course, can be covered by Lemma \ref{inj-radius}. A more important fact is that Perelman's entropy and the Nash entropy on the ancient solution converge to the entropy of the asymptotic shrinker; see Proposition \ref{Asymptotic_Shrinker} below. Let us first review these notions of entropy.

Let $(M,g(t))_{t\in(-\infty,w)}$ be an ancient solution, where $0<w\leq\infty$. Let $(x_0,t_0)\in M\times(-\infty,w)$ be a fixed point in space-time and $u: M\times (-\infty,t_0)\rightarrow\mathbb{R}_+$ the fundamental solution to the conjugate heat equation $\displaystyle -\partial_t u-\Delta u+Ru=0$ based at $(x_0,t_0)$. We write $u$ as
\begin{eqnarray*}
u:=(4\pi\tau)^{-\frac{n}{2}}e^{-f},
\end{eqnarray*}
where $\tau=t_0-t\in(0,\infty)$. Then Perelman's entropy and the Nash entropy based at $(x_0,t_0)$ are respectively defined as
\begin{eqnarray*}
\mathcal{W}_{x_0,t_0}(\tau)&:=&\int_M\Big(\tau\big(|\nabla f|^2+R\big)+f-n\Big)udg_t,
\\
\mathcal{N}_{x_0,t_0}(\tau)&:=&\int_M fudg_t-\frac{n}{2}.
\end{eqnarray*}
It is a well known fact that both $\mathcal{W}_{x_0,t_0}(\tau)$ and $\mathcal{N}_{x_0,t_0}(\tau)$ are increasing in $t$ (or decreasing in $\tau$). Furthermore, we have
\begin{eqnarray*}
\lim_{\tau\rightarrow 0+}\mathcal{W}_{x_0,t_0}(\tau)=\lim_{\tau\rightarrow 0+}\mathcal{N}_{x_0,t_0}(\tau)=0.
\end{eqnarray*}
The following result is already well-established in literature.

\begin{Proposition}\label{Asymptotic_Shrinker}
Let $(M,g(t))_{t\in(-\infty,w)}$, where $0<w\leq\infty$, be an ancient solution as described in Theorem \ref{MainTheorem1} or Theorem \ref{MainTheorem2}. Let $(x_0,t_0)\in M\times(-\infty,w)$ be an arbitrarily fixed point and $\displaystyle u:=(4\pi\tau)^{-\frac{n}{2}}e^{-f}$  the conjugate heat kernel based at $(x_0,t_0)$. Let $\tau_i\nearrow\infty$ be an increasing sequence of positive numbers. Then, the following sequence of tuples
\begin{eqnarray*}
\Big\{\big(M,g_i(t),(x_0,-1),f_i\big)_{t\in(-\infty,0)}\Big\}_{i=1}^\infty
\end{eqnarray*}
converges, possibly after passing to a subsequence, to the canonical form of a shrinking gradient Ricci soliton, called the \emph{asymptotic shrinker}
\begin{eqnarray*}
\big(M_\infty,g_\infty(t),(x_\infty,-1),f_\infty\big)_{t\in(-\infty,0)},
\end{eqnarray*}
where $f_\infty$ is the potential function, satisfying 
\begin{eqnarray*}
Ric_{g_\infty}+\nabla^2f_\infty=\frac{1}{-2t}g_\infty.
\end{eqnarray*}
The convergence $g_i\rightarrow g_\infty$ is in the pointed Cheeger-Gromov-Hamilton \cite{Ha3} sense, and the convergence $f_i\rightarrow f_\infty$ is in the locally smooth sense. Here $g_i$ and $f_i$ are obtained by time-shifting and parabolic scaling as follows
\begin{eqnarray*}
g_i(t)&:=&\tau_i^{-1}g(\tau_it+t_0),
\\
f_i(\cdot,t)&:=&f(\cdot,\tau_it+t_0).
\end{eqnarray*}
Furthermore, Perelman's entropy and the Nash entropy converge to the entropy of the asymptotic shrinker. By this we mean
\begin{eqnarray}\label{eq1}
\int_M(4\pi|t|)^{-\frac{n}{2}}e^{-f_\infty}dg_\infty=1,
\\\nonumber
\lim_{\tau\rightarrow \infty}\mathcal{W}_{x_0,t_0}(\tau)=\lim_{\tau\rightarrow \infty}\mathcal{N}_{x_0,t_0}(\tau)=\mu_\infty,
\end{eqnarray}
where
\begin{eqnarray}\label{eq2}
\mu_\infty&=&\int_M \Big(|t|(|\nabla f_\infty|^2+R_{g_\infty})+f_\infty-n\Big)(4\pi|t|)^{-\frac{n}{2}}e^{-f_\infty}dg_\infty
\\\nonumber
&=&\int_M f_\infty(4\pi|t|)^{-\frac{n}{2}}e^{-f_\infty}dg_\infty-\frac{n}{2}
\end{eqnarray}
is a negative constant independent of time $t$, which we call \emph{the entropy of the asymptotic shrinker}.
\end{Proposition}

\begin{proof}
The proof of this proposition can be modified from, for instance, \cite{CZ} or \cite{X}. First of all, the Type I condition implies that there exists a positive number $C$, such that
\begin{eqnarray}\label{tp1}
|Rm_{g(t)}|\leq\frac{C}{|t|}
\end{eqnarray} 
for all $t\in(-\infty,0)$. We then observe that for all $(x,t)\in M\times (-\infty,0)$, it holds that
\begin{eqnarray}\label{vol}
\operatorname{Vol}_{g(t)}\big(B_{g(t)}(x,\sqrt{|t|})\big)\geq c|t|^{\frac{n}{2}},
\end{eqnarray}
where $c>0$ is a constant. Suppose this is not true, then one may find a sequence of counterexamples $\{(x_i,t_i)\}_{i=1}^\infty\subset M\times(-\infty,0)$, such that the scaled Ricci flows $(M,g_i(t))_{t\in(-\infty,0)}$, where $g_i(t):=|t_i|^{-1}g_i(t|t_i|)$, all satisfy (\ref{tp1}), but
\begin{eqnarray}\label{vol2}
\operatorname{Vol}_{g_i(-1)}\big(B_{g_i(-1)}(x_i,1)\big)\rightarrow 0.
\end{eqnarray}
By Lemma \ref{inj-radius}, we have 
\begin{eqnarray*}
\operatorname{inj}(g_i(-1),x_i)\geq c>0.
\end{eqnarray*}
Hence, by \cite{Ha3}, the sequence of Ricci flows $\displaystyle \big\{(M,g_i(t),(x_i,-1))_{t\in(-\infty,0)}\big\}_{i=1}^\infty$ converges, possibly after passing to a subsequence, to a smooth ancient solution $(M_\infty,g_\infty(t),(x_\infty,-1))_{t\in(-\infty,0)}$. In particular, we have
\begin{eqnarray*}
\operatorname{Vol}_{g_\infty(-1)}\big(B_{g_\infty(-1)}(x_\infty,1)\big)>0,
\end{eqnarray*}
and this contradicts (\ref{vol2}).

One may then follow the proofs in \cite{CZ} or \cite{X}  to conclude this Proposition. Obviously, the noncollapsing condition in these proofs can be replaced by (\ref{vol}).
 
\end{proof}

\section{The Nash Entropy and Noncollapsing}

It turns out that from Proposition \ref{Asymptotic_Shrinker} it is sufficient to conclude that the ancient solution is noncollapsed on all scales everywhere. This follows from an observation made in \cite{MZ}. Let $(M,g(t))_{t\in(-\infty,w)}$, where $0<w\leq\infty$, be an ancient solution as describe in Theorem \ref{MainTheorem1} or Theorem \ref{MainTheorem2}. One may generally regard $(-\infty,w)$ as the maximum existing interval of $g(t)$, in which case $t=w$ is the singular time (whether it is infinity or not). The following lemma says $g(t)$ has bounded geometry as long as it is regular.

\begin{Lemma}
For all $t\in(-\infty,w)$, it holds that $$\sup_M\big|Rm_{g(t)}\big|<\infty\text{ and }\displaystyle\inf_{x\in M}\operatorname{Vol}_{g(t)}\big(B_{g(t)}(x,1)\big)>0.$$
\end{Lemma}
\begin{proof}
If $w<\infty$, then by the definition of finite singular time, this is the first instance at which $g(t)$ has unbounded curvature. If $w=\infty$, then this means that $g(t)$ has bounded curvature for all $t$. The volume lower bound for unit balls follows from a straightforward volume distortion estimate.
\end{proof}

From this time-wise geometry bound, we may conclude the following proposition; this is a combination of Proposition \ref{Asymptotic_Shrinker} above and Proposition 4.6 in \cite{MZ}, where the second author together with Zilu Ma proved (as a consequence of Corollary 5.11 in \cite{Ba}) that on an ancient solution with bounded geometry on each time-slice, Perelman's entropies and the Nash entropies based at all points converge to the same number as the time approaches negative infinity.

\begin{Proposition}\label{globalentropy}
Let $(M,g(t))_{t\in(-\infty,w)}$, where $0<w\leq\infty$, be an ancient solution as describe in Theorem \ref{MainTheorem1} or Theorem \ref{MainTheorem2}. Then for all $(x,t)\in M\times(-\infty,w)$, the following holds
\begin{eqnarray*}
\lim_{\tau\rightarrow\infty}\mathcal{W}_{x,t}(\tau)=\lim_{\tau\rightarrow\infty}\mathcal{N}_{x,t}(\tau)=\mu_\infty,
\end{eqnarray*}
where $\mu_\infty$ is the entropy of any one of asymptotic shrinkers based at any point (for their entropies are all equal). In particular, we have
\begin{eqnarray}
\mathcal{N}_{x,t}(\tau)\geq \mu_\infty
\end{eqnarray}
for all  $(x,t)\in M\times(-\infty,w)$ and for all $\tau>0$.
\end{Proposition}

Though in the original proof of Perelman \cite{P}, he uses the bound of the $\mu$ functional to show the no local noncollapsing theorem, yet the second author showed that the boundedness of the Nash entropy could also be used to prove the noncollapsedness at its base point; this is the following Proposition.

\begin{Proposition}[Theorem 6.1 in \cite{Ba}]\label{noncollapsing}
Let $(M,g(t))$ be a Ricci flow and $(x,t)$ a point in the space time. Let $r$ be a positive scale such that $[t-r^2,t]$ is in the existing interval and $R\leq r^{-2}$ on $\displaystyle B_{g(t)}(x,r)$. Then, it holds that
\begin{eqnarray*}
\operatorname{Vol}_{g(t)}\big(B_{g(t)}(x,r)\big)\geq c\exp\big(\mathcal{N}_{x,t}(r^2)\big)r^n.
\end{eqnarray*}
Here $c$ is a positive constant depending only on the dimension.
\end{Proposition}

\emph{Remark:} Though Bamler \cite{Ba} proved the above result for Ricci flows on closed manifolds, yet one may check the proof of Theorem 6.1 in \cite{Ba} and easily verify its validity for Ricci flows with bounded geometry on each time-slice; one may need to apply Theorem 4.4 of \cite{MZ} in this verification. Fortunately, all the Ricci flows we work with in this paper satisfy this condition. On the other hand, the second author proved that bounded Nash entropy implies weak noncollapsing, and this proof does not need bounded geometry on each time-slice; see Proposition 3.3 in \cite{Z1}.

\bigskip

\begin{proof}[Proof of Theorem \ref{MainTheorem1} and Theorem \ref{MainTheorem2} ]
Let $(M,g(t))_{t\in(-\infty,w)}$ be an ancient solution as described in either Theorem \ref{MainTheorem1} or Theorem \ref{MainTheorem2}. Let $(x,t)\in M\times(-\infty,w)$ be an arbitrary space-time point, and $r$ any scale that satisfies
\begin{eqnarray*}
R\leq r^{-2} \text{ on } B_{g(t)}(x,r).
\end{eqnarray*}

Since, by Proposition \ref{globalentropy}, we have
\begin{eqnarray*}
\mathcal{N}_{x,t}(r^2)\geq \mu_\infty\in(-\infty,0),
\end{eqnarray*}
where $\mu_\infty$ is the entropy of one of the asymptotic shrinkers, it then follows from Proposition \ref{noncollapsing} that
\begin{eqnarray*}
\operatorname{Vol}_{g(t)}\big(B_{g(t)}(x,r)\big)\geq ce^{\mu_\infty} r^n;
\end{eqnarray*}
this finishes the proof.

\end{proof}

\section{Applications}

In this section, we prove all the corollaries proposed in the introduction section. 

\begin{proof}[Proof of Corollary \ref{3d}]
Let $(M^3,g(t))$ be a three-dimensional noncompact Type I ancient solution. By Chen \cite{C}, $g(t)$ has nonnegative sectional curvature everywhere. If its sectional curvature is strictly positive, then, by Theorem \ref{MainTheorem1}, it is also noncollapsed. It follows from \cite{Z2} (or \cite{H}) that such ancient solution does not exist.

If $g(t)$ ever attains zero sectional curvature somewhere, then by the strong maximum principle of Hamilton \cite{Ha2},  $(M^3,g(t))$ splits locally and hence its universal cover must be the standard shrinking cylinder $\mathbb{S}^2\times\mathbb{R}$. Furthermore, the only noncompact quotients of $\mathbb{S}^2\times\mathbb{R}$ are the $\mathbb{Z}_2$ quotients. The reason is that the projection of a group action on the $\mathbb{R}$ factor can only be either the reflection or the identity---if it is ever a translation, then this action will generate an infinity group action on $\mathbb{S}^2\times\mathbb{R}$, and the quotient space must be compact. This finishes the proof of the corollary.
\end{proof}

\begin{proof}[Proof of Corollary \ref{noncompact}]
We argue by contradiction. Assume one of the asymptotic shrinkers is the standard cylinder $\mathbb{S}^{n-1}\times\mathbb{R}$, then all the results obtained in section 3 of \cite{LZ} can be applied to this ancient solution. In particular, it satisfies a canonical neighborhood theorem and hence always has a non-neck-like region at each time (c.f. Theorem 1.3 in \cite{LZ}), it always splits as $\mathbb{S}^{n-1}\times\mathbb{R}$ at space infinity (c.f. Proposition 3.9 in \cite{LZ}), it satisfies the neck stability theorem of Kleiner-Lott \cite{KL} (c.f. Theorem 3.11 in \cite{LZ}), and all such ancient solutions form a compact space (c.f. Theorem 1.2 in \cite{LZ}). Knowing all these facts, one may follow the arguments in \cite{Z2} line by line to conclude that such an ancient solution does not exist.
\end{proof}

\begin{proof}[Proof of Corollary \ref{compact}]
The $\operatorname{PIC}-2$ condition implies the positive sectional curvature condition. Hence, according to Theorem \ref{MainTheorem2}, such ancient solution must be $\kappa$-noncollapsed on all scales for some $\kappa>0$. The conclusion then follows from Corollary 0.4 in \cite{Ni2}.
\end{proof}

\noindent School of Mathematics and Statistics \& Hubei Key Laboratory of Mathematical Sciences, Central China Normal University, Wuhan, 430079, P.R.China
\\ E-mail address: \verb"chengliang@mail.ccnu.edu.cn"
\\

\noindent School of Mathematics, University of Minnesota, Twin Cities, MN, 55414, USA
\\ E-mail address: \verb"zhan7298@umn.edu"

\end{document}